\documentclass[12pt]{amsart}
\usepackage{amsfonts}
\usepackage{amssymb,amscd,pstricks}
\usepackage[active]{srcltx}
\usepackage[all]{xy}
\title [A new inequality on the Hodge number $h^{1,1}$ ]
{ A new inequality on the Hodge number $h^{1,1}$ of algebraic surfaces}
\author[Jun Lu]{Jun Lu}
\author[Sheng-Li Tan]{Sheng-Li Tan}
\author[Fei Yu]{Fei Yu}
\author[Kang Zuo]{Kang Zuo}
\address{Department of Mathematics, East China Normal University,   Dongchuan RD 500,
 Shanghai 200241, P. R. of China}
 \email{jlu@math.ecnu.edu.cn  }
 \address{Department of Mathematics, East China Normal University, Dongchuan RD 500,
 Shanghai 200241, P. R. of China}
 \email{sltan@math.ecnu.edu.cn}
\address{Department of Mathematics, Xiamen University, Xiamen
361005, P. R. of China} \email{vvyufei@gmail.com}
\address{Universit\"{a}t Mainz, Fachbereich 08-Physik Mathematik und Informatik, 55099 Mainz, Germany}
\email{zuok@uni-mainz.de}

\thanks{}
 \fussy
\begin{document}
\theoremstyle{plain}
\newtheorem{thm}{Theorem}[section]
\newtheorem{theorem}[thm]{Theorem}
\newtheorem{addendum}[thm]{Addendum}
\newtheorem{lemma}[thm]{Lemma}
\newtheorem{corollary}[thm]{Corollary}
\newtheorem{proposition}[thm]{Proposition}
\theoremstyle{definition}
\newtheorem{remark}[thm]{Remark}
\newtheorem{remarks}[thm]{Remarks}
\newtheorem{notations}[thm]{Notations}
\newtheorem{definition}[thm]{Definition}
\newtheorem{claim}[thm]{Claim}
\newtheorem{assumption}[thm]{Assumption}
\newtheorem{assumptions}[thm]{Assumptions}
\newtheorem{property}[thm]{Property}
\newtheorem{properties}[thm]{Properties}
\newtheorem{example}[thm]{Example}
\newtheorem{examples}[thm]{Examples}
\newtheorem{conjecture}[thm]{Conjecture}
\newtheorem{questions}[thm]{Questions}
\newtheorem{question}[thm]{Question}
\numberwithin{equation}{section}
\newcommand{\sA}{{\mathcal A}}
\newcommand{\sB}{{\mathcal B}}
\newcommand{\sC}{{\mathcal C}}
\newcommand{\sD}{{\mathcal D}}
\newcommand{\sE}{{\mathcal E}}
\newcommand{\sF}{{\mathcal F}}
\newcommand{\sG}{{\mathcal G}}
\newcommand{\sH}{{\mathcal H}}
\newcommand{\sI}{{\mathcal I}}
\newcommand{\sJ}{{\mathcal J}}
\newcommand{\sK}{{\mathcal K}}
\newcommand{\sL}{{\mathcal L}}
\newcommand{\sM}{{\mathcal M}}
\newcommand{\sN}{{\mathcal N}}
\newcommand{\sO}{{\mathcal O}}
\newcommand{\sP}{{\mathcal P}}
\newcommand{\sQ}{{\mathcal Q}}
\newcommand{\sR}{{\mathcal R}}
\newcommand{\sS}{{\mathcal S}}
\newcommand{\sT}{{\mathcal T}}
\newcommand{\sU}{{\mathcal U}}
\newcommand{\sV}{{\mathcal V}}
\newcommand{\sW}{{\mathcal W}}
\newcommand{\sX}{{\mathcal X}}
\newcommand{\sY}{{\mathcal Y}}
\newcommand{\sZ}{{\mathcal Z}}
\newcommand{\A}{{\mathbb A}}
\newcommand{\B}{{\mathbb B}}
\newcommand{\C}{{\mathbb C}}
\newcommand{\D}{{\mathbb D}}
\newcommand{\E}{{\mathbb E}}
\newcommand{\F}{{\mathbb F}}
\newcommand{\G}{{\mathbb G}}
\newcommand{\BH}{{\mathbb H}}
\newcommand{\I}{{\mathbb I}}
\newcommand{\J}{{\mathbb J}}
\newcommand{\BL}{{\mathbb L}}
\newcommand{\M}{{\mathbb M}}
\newcommand{\N}{{\mathbb N}}
\newcommand{\BP}{{\mathbb P}}
\newcommand{\Q}{{\mathbb Q}}
\newcommand{\R}{{\mathbb R}}
\newcommand{\BS}{{\mathbb S}}
\newcommand{\T}{{\mathbb T}}
\newcommand{\U}{{\mathbb U}}
\newcommand{\V}{{\mathbb V}}
\newcommand{\W}{{\mathbb W}}
\newcommand{\X}{{\mathbb X}}
\newcommand{\Y}{{\mathbb Y}}
\newcommand{\Z}{{\mathbb Z}}
\newcommand{\rk}{{\rm rk}}
\newcommand{\ch}{{\rm c}}
\newcommand{\Sp}{{\rm Sp}}
\newcommand{\Sl}{{\rm Sl}}
\maketitle

\footnotetext[1]{This work was supported by the SFB/TR 45  Periods,
Moduli Spaces and Arithmetic of Algebraic Varieties  of the DFG
(German Research Foundation).}

\footnotetext[2]{The first and the second named authors are also
supported by NSFC, the Science  Foundation  of  the EMC  and the
Foundation of Scientific Program of Shanghai. }

\footnotetext[3]{The first author is also supported by the Fundamental Research Funds for the Central Universities. }

 \maketitle
\begin{abstract} We get a new inequality on the Hodge number $h^{1,1}(S)$ of fibred algebraic complex surfaces $S$,
which is a generalization of an inequality of Beauville. Our
inequality  implies the Arakelov type inequalities due to  Arakelov,
Faltings, Viehweg and Zuo, respectively.
\end{abstract}

\qquad
  \section{Introduction and main results}
For a compact complex K\"{a}hler surface $S$, we have several Hodge
numbers
$$h^{0,1}(S)=h^{1,0}(S)=q(S), \hskip0.5cm
h^{0,2}(S)=h^{2,0}(S)=p_g(S), \hskip0.5cm h^{1,1}(S).$$
 However, the
Hodge number $h^{1,1}$ is not well understood comparing with the
others. Lefschetz's (1,1)-theorem tells us that the N\'{e}ron-Severi
group  NS$(S) =H^{1,1}(S)\cap H^2(S,\mathbb Z)$. Denote by
$\rho(S)=\textrm{rank} {\rm\,NS}(S)$ the Picard number of $S$, i.e.,
the rank of NS$(S)$. Then we have $h^{1,1}(S)\geq \rho(S)$.

\begin{theorem}[\cite{BPV04}, Corollary~5.4]
If a compact K\"{a}hler surface $S$ does not admit any fibration
with connected fibers over a curve of genus $b\geq 2$, then
\begin{equation}\label{BPV} h^{1,1}(S)\geq 2q(S)-1.
\end{equation}
\end{theorem}

 In what follows, we consider the case when $S$ admits a fibration $f:S\to C$ over a smooth curve $C$ of genus
 $b$. It is obvious
that $q(S)\geq b$. We assume that the fibers are connected and the generic fiber
 is a smooth curve of genus $g$.
Let $F_1,\cdots,F_s$ be all singular fibers of $f$ and $\ell(F_i)$
be the number of irreducible components of $F_i$.
\begin{theorem}[\cite{Bea81}, Lemma~2]
\begin{align}\label{BEA}
h^{1,1}(S)\geq \rho(S)\geq  2+\sum\limits_{i=1}^s(\ell(F_i)-1).
\end{align}
\end{theorem}

In this paper, we will generalize Beauville's inequality.

   \begin{theorem}\label{theoremh11} Let $f:S\to C$ be a genus $g$ fiberation as above, $b=g(C)$. Then
\begin{align}\label{eqh11}  h^{1,1}(S)\geq
2(q(S)-b)b+2+\sum\limits_{i=1}^s(\ell(F_i)-1).
\end{align}
\end{theorem}
As a consequence,
$$h^{1,1}(S)\geq 2bq(S)+2-2b,$$
it is an analogue of
\eqref{BPV} for fibered surfaces.

Let $\Gamma_1,\cdots, \Gamma_k$ be all irreducible components of a
given fiber $F$, and $\widetilde{\Gamma}_i\to \Gamma_i$ be the
normalization of $\Gamma_i$.
   $g(F):=\sum\limits_{i=1}^k g(\widetilde{\Gamma}_i)$ is called the geometric genus of $F$.  One can see that  $g(F)\leq
   p_a(F)=g$.

   Denote by $q_f:=q(S)-b$ the relative irregularity of $f$.
Beauville \cite{Bea81} proves  that for any fiber $F$,
\begin{equation}\label{BEA2}
g(F)\geq q_f\geq 0.\end{equation}

   Let $s_1$ be the number of singular fibers satisfying
 $g(F)< g$, without loss of generality, we assume that $F_1,\cdots,F_{s_1}$ satisfy $g(F)<g$. In fact, $s_1$ is exactly
 the number of singular fibers of the associated family of Jacobians of the fibers.
The new inequality on $h^{1,1}$ is related to the Arakelov
inequality due to Viehweg and the last author. 
 \begin{theorem}\label{theoremequality}
With the notations as above,  we have
\begin{align*} 2\deg
f_*\omega_{S/C}  = \, &
(g-q_f)(2b-2+s_1)-\sum\limits_{i=1}^{s_1}(g(F_i)-q_f)
\\
&-\left(h^{1,1}(S)
-2q_fb-2-\sum\limits_{i=1}^s(\ell(F_i)-1)\right)+\sum\limits_{i=1}^{s_1}N_{\overline{F}_i}
.
\end{align*}
where $g\geq N_{\overline{F}_i}:=g-p_a(\overline F_{i,\,\text{\rm red}})\geq 0$,
and $\overline F_i=\sigma^*F_i$ is the normal crossing model of
$F_i$ obtained by a sequence of blowing-ups, see Sect.
\ref{Secaraeqn}. In particular, if $F_i$ is semistable, then
$N_{\overline{F}_i}=0$.
  \end{theorem}

 As an application of the inequalities \eqref{eqh11}, \eqref{BEA2} and Theorem \ref{theoremequality},
 we have
\begin{corollary}\label{theoremAras1} {\rm (Viehweg-Zuo
\cite{VZ06})} Let $f:S\to C$ be a non-trivial semistable fibration.
Let $s_1$ be the number of fibers with non-compact Jacobians. Then we
have
\begin{align}\label{eqarasemi}
 \deg f_*\omega_{S/C}\leq \frac{g-q_f}{2}(2b-2+s_1).
\end{align}
\end{corollary}
Furthermore, Viehweg and Zuo prove that the above inequality becomes
equality if and only if the curve $C$ is a Shimura curve in the
moduli space of curves.

Since $s_1\leq s$, the above result implies the classical Arakelov inequality as follows.
\begin{corollary}[Arakelov \cite{Ara71}, Faltings \cite{Fal83}]
$$\deg f_*\omega_{S/C}\leq \frac{g-q_f}{2}(2b-2+s).$$
In particular,   we have a  weaker inequality
\begin{align}\label{eqweakineq}
\deg f_*\omega_{S/C}\leq \frac{g}{2}(2b-2+s).
\end{align}
\end{corollary}
\begin{remark}
In fact the inequality \eqref{eqweakineq} is strict when $g\geq 2$.  One can also get it by  combining
 Cornalba-Harris-Xiao's inequality \cite{CH88, Xia87}
$$
 \dfrac{4g-4}{g}\deg f_*\omega_{S/C}\leq \omega_{S/C}^2
$$
with Vojta's canonical class inequality \cite{Voj88}
\begin{align}
\omega_{S/C}^2 <  (2g-2) \cdot (2b-2+s),
\end{align}
which is indeed strict (cf. \cite{Tan95, Liu96}).

\end{remark}

When $f:S\to\mathbb P^1$ is a semistable family  over $\mathbb P^1$,
the equality in Theorem \ref{theoremequality} can be rewritten as
\begin{align}\label{pg=} \frac12(g-q(S))(s_1-4)= \ & p_g(S)+\dfrac12\sum\limits_{i=1}^{s_1}\left(g(F_i)-q(S)\right)
\\ \notag 
&+\dfrac12\left(h^{1,1}(S)-2-\sum\limits_{i=1}^s(\ell(F_i)-1)\right),
\end{align}
where the formula $\deg f_*\omega_{S/C}=\chi(\mathcal
O_S)-(g-1)(b-1)$ is used.
\begin{corollary}\label{corPgq1s14}
Assume that  $f:S\to \BP^1$ is semistable and non-trivial. Then we
have $s_1\geq 4$. $s_1=4$  if and only if the following conditions
are satisfied.
\begin{enumerate}
\item[$(1)$] \ $p_g(S)=0$;   $g(F_i)=q(S)$ for $i=1,\cdots,4$;

\item[$(2)$] \ $h^{1,1}(S)=2+\sum\limits_{i=1}^s (\ell(F_i)-1)$.
\item[$(3)$] \ $q(S)\leq 1$;
\end{enumerate}
\end{corollary}
(1) and (2) are direct consequences of \eqref{BEA} and \eqref{BEA2}.
(3) will be proved in \S4. In \S5, we will present two examples of
genus $g=2$ with $s_1=4$ and $q(S)=1$.

 In  the case when $f$ is non-semistable, we have
 $$N_{\overline F_i}=g-p_a(\overline F_{i,\rm red})\leq g-g(\overline F_i)=g-g(F_i)\leq g-q_f.$$
 We get an inequality as follows.
\begin{corollary}\label{theoremAra2}
If $f$ is non-semistable, then
\begin{align*}
\deg f_*\omega_{S/C}\leq (g-q_f)(b-1+s_1).
\end{align*}
\end{corollary}

\qquad

\section{Proof of Theorem \ref{theoremh11}}\label{Secaraeqn}
Let $f:S\to C$ be a fibration of genus $g$ over a smooth curve $C$
of genus $b$, and $F_1,\cdots, F_s$ be all singular fibers. Due to
Beauville's inequality \eqref{BEA}, we can assume that $b>0$ and
$q_f>0$.  Consider $H^0(\Omega_S)=V_1\oplus V_0$, where
$V_0=f^*H^0(\Omega_C)$ and $\dim V_1=q_f$. Let
 $$V_0=\langle\alpha_1\cdots \alpha_b\rangle, \,\,\,\, V_1=\langle \theta_1,\cdots, \theta_{q_f}\rangle,$$
where $\alpha_i$'s (resp. $\theta_j$'s ) are the base of $V_0$
(resp. $V_1$) as a $\mathbb C$-vector space.

We define a homomorphism
$$h: V_0\otimes \overline{V}_1\oplus \overline{V}_0\otimes V_1\longrightarrow H^{1,1}(S),$$
by $h(x\otimes y)=x\wedge y$ for $x\otimes y\in  V_0\otimes
\overline{V}_1\oplus \overline{V}_0\otimes V_1$.

Let $V_2$ be the subgroup of ${\rm Pic(S)}\otimes\mathbb R$ generated by the classes
of the components of all fibers. Chern class induces a homomorphism
 $c_1: V_2\to H^{1,1}(S)\cap H^2(S,\mathbb R)$. By the semi-negativity of the intersection matrix of a fiber, we have
 $${\rm{dim\,Im}} (c_1)=1+\sum\limits_{i=1}^s(\ell(F_i)-1).$$

\begin{lemma}\label{Claim 1.}  For any ample divisor $H$,
$c_1(H)\notin {\rm Im\,} h + ({\rm Im\,} c_1){\otimes_{\mathbb R}\mathbb C}.$
\end{lemma}
\begin{proof}
Suppose that $c_1(H)=\alpha+\beta\in {\rm Im\,} h+({\rm Im\,}
c_1)\otimes_{\mathbb R}\mathbb C$ for some  $\alpha\in {\rm Im\,}
h$. Let $F$ be a general fiber. For any $\alpha\in {\rm Im\,} h$, by
the definition of $h$, one can see easily that $\alpha|_F=0$. On the
other hand, Zariski's lemma implies $\beta|_F=0$ for any $\beta\in
{\rm Im\,} c_1$. Hence $c_1(H)|_F=0$, i.e.,  $HF=0$, a
contradiction.
\end{proof}

Note that $\alpha_i\wedge \bar{\alpha}_k$ is the pull-back of an
element in $H^{1,1}(C)$  since $h^1(C, \Omega_C)= h^0(C, \sO_C)=1$.
So we can assume that $\alpha_i\wedge
\bar{\alpha}_k=\varepsilon_{ik}\alpha_1\wedge \bar{\alpha}_1$ as
cohomology classes in $H^{1,1}(C)$, where $\varepsilon_{ik}$ are
complex numbers.

\begin{lemma}\label{Claim 2.}  The matrix $E=(\varepsilon_{ik})_{1\leq j,\,k\leq b}$ is invertible.
\end{lemma}
\begin{proof}
Suppose that there is a vector $(\lambda_1,\cdots, \lambda_b)\ne 0$,
such that $E\cdot (\lambda_1,\cdots, \lambda_b)^T=0$, i.e.,
$\sum\limits_k \varepsilon_{ik}\lambda_k=0$ for all $i$. Then
$\sum\limits_k \varepsilon_{ik}\lambda_k\alpha_1\wedge
\bar{\alpha}_1=0$. Namely, one has
$$\sum\limits_k\lambda_k\alpha_i\wedge \bar{\alpha}_k=\alpha_i\wedge\sum\limits_{k}\lambda_k\bar{\alpha}_k=0,\,\, \textrm{ for all } i. $$
Therefore we get $\sum\limits_{k}\bar{\lambda}_k \alpha_k\wedge \sum\limits_{k}\lambda_k\bar{\alpha}_k=0$, that is,
$\sum\limits_{k}\bar{\lambda_k}\alpha_k=0$. So $\lambda_i=0$ for all $i$, a contradiction.
\end{proof}

\begin{lemma}\label{Claim 3.} $h$ is injective.
\end{lemma}
\begin{proof} Suppose that there is a nonzero element in the kernel of
$h$,
\begin{equation}
\sum\limits_{j=1}^{q_f}\sum\limits_{i=1}^{b}a_{ij}\alpha_i\wedge
\bar{\theta}_j
+\sum\limits_{j=1}^{q_f}\sum\limits_{i=1}^{b}b_{ij}\bar{\alpha}_i\wedge
\theta_j=du.
\end{equation}
 (A Zero cohomology class means an exact form). Note that $d\alpha_k=d\theta_l=0$, $
 \bar\alpha_i\wedge \bar\alpha_k=0$.
    By wedging $\bar{\alpha}_k\wedge
\theta_l$ on both sides, one gets
\begin{equation}\label{eq2.2} \sum\limits_{j=1}^{q_f}\sum\limits_{i=1}^{b}a_{ij}\alpha_i\wedge \bar{\theta}_j\wedge \bar{\alpha}_k\wedge \theta_l=
 d(u\wedge \bar{\alpha}_k\wedge \theta_l).\end{equation}
So
\begin{equation}\sum\limits_{j=1}^{q_f}\sum\limits_{i=1}^{b}a_{ij}\varepsilon_{ik}\alpha_1\wedge\bar{\alpha}_1\wedge\bar{\theta}_j\wedge \theta_l
=d(-u\wedge\bar{\alpha}_k\wedge \theta_l).\end{equation}
 Let
$\omega_k=\sum\limits_{j=1}^{q_f}\sum\limits_{i=1}^{b}\bar{a}_{ij}\bar{\varepsilon}_{ik}
\theta_j$. We have $\alpha_1\wedge\bar{\alpha}_1\wedge
\bar{\omega}_k\wedge \theta_l=d(-u\wedge\bar{\alpha}_k\wedge
\theta_l)$. It implies that
$$\alpha_1\wedge\bar{\alpha}_1\wedge \bar{\omega}_k\wedge\omega_k=d(-u\wedge\alpha_k\wedge\omega_k).$$

By Stokes formula,
$$0=\int_S \alpha_1\wedge\bar{\alpha}_1\wedge \bar{\omega}_k\wedge\omega_k=\int_S
(\alpha_1\wedge \omega_k)\wedge \overline{(\alpha_1\wedge
\omega_k)} \ .$$ So $\alpha_1\wedge\omega_k=0$, i.e.,
$\omega_k=f^*\beta_k$ for some $\beta_k\in H^0(C, \Omega_C)$. Thus
$\omega_k\in V_0\cap V_1$, i.e., for any $k$, $\omega_k=0$.
 Hence
$$\sum\limits_{i}a_{ij}\varepsilon_{ik}=0,\hskip1cm \text{\rm for any } j \ \text{\rm and }  k.$$
Therefore Lemma \ref{Claim 2.} implies that  $a_{ij}=0$ for all $i$
and $j$. Similarly, we have $b_{ij}=0$. It is a contradiction.
\end{proof}

\begin{lemma}\label{Claim 4.} ${\rm Im\,} h\cap ({\rm Im\,} c_1)\otimes_{\mathbb R}\mathbb C =0$.
\end{lemma}
\begin{proof}
Note that ${\rm Im\,} c_1 \subseteq H^{1,1}(S)\cap H^2(S,\mathbb R)$.
Let $\Gamma_{i1}$, $\cdots$, $\Gamma_{i\ell_i}$ be the irreducible components of $F_i$,
let $\omega=c_1(F)$ and  let $\omega_{ij}=c_1(\Gamma_{ij})$. Assume
that
\begin{equation}
x\cdot\omega+\sum_{i=1}^s\sum_{j=1}^{\ell_i-1}
x_{ij}\cdot\omega_{ij}=t \in {\rm Im\,} h,  \qquad x, \ x_{ij} \in
\mathbb C.
\end{equation}
 Note that for any component $\Gamma$ in the fibers,
$t|_{\Gamma}=0$ because $\alpha_i$'s are pullback of forms on the
base $C$. Similarly, $\omega|_{\Gamma}=0$.
\begin{equation}
\sum_{i=1}^s\sum_{j=1}^{\ell_i-1}
x_{ij}\cdot\omega_{ij}|_{\Gamma}=0.
\end{equation}
Thus
\begin{equation}
\sum_{i=1}^s\sum_{j=1}^{\ell_i-1}
x_{ij}\cdot\int\omega_{ij}|_{\Gamma}=0.
\end{equation}
Namely,
\begin{equation}
\sum_{i=1}^s\sum_{j=1}^{\ell_i-1} x_{ij}\cdot \Gamma_{ij}\Gamma=0
\hskip0.5cm \textrm{ for any } \Gamma.
\end{equation}
 We know that the intersection matrix of $\Gamma_{ij}$ is
negative definite, for $1\leq i \leq s$ and $1 \leq j \leq
\ell_i-1$\,, we have $x_{ij}=0$. Thus we get that $x\cdot c_1(F)=t
\in {\rm Im\,} h$ for some $x \in \mathbb C$.

If $x\neq 0$, then $c_1(F) \in {\rm Im\,} h$.  Note that
$c_1(F)=c\cdot\alpha_1\wedge\bar\alpha_1\neq 0$, i.e., $c\neq 0$.
 Let
 $$
c_1(F)=\sum\limits_{j=1}^{q_f}\sum\limits_{i=1}^{b}a_{ij}\alpha_i\wedge
\bar{\theta}_j
+\sum\limits_{j=1}^{q_f}\sum\limits_{i=1}^{b}b_{ij}\bar{\alpha}_i\wedge
\theta_j\ ,
 $$
As forms, we have
\begin{equation}
\sum\limits_{j=1}^{q_f}\sum\limits_{i=1}^{b}a_{ij}\alpha_i\wedge
\bar{\theta}_j
+\sum\limits_{j=1}^{q_f}\sum\limits_{i=1}^{b}b_{ij}\bar{\alpha}_i\wedge
\theta_j - c\cdot \alpha_1\wedge\bar\alpha_1=du.
\end{equation}
Similar to the proof of the previous lemma, we get also
\eqref{eq2.2}. Use the same proof as above, we get $a_{ij}=b_{ij}=0$
for any $i$ and $j$. Hence $ -c\cdot\alpha_1\wedge\bar\alpha_1=du $.
Thus as a class, $c_1(F)=c\cdot\alpha_1\wedge\bar\alpha_1=0$, a
contradiction.
 This proves the lemma
\end{proof}

Combing the above lemmas, we have
$$h^{1,1}\geq 1+{\rm dim \,
Im\,}h+{\rm dim\,} \left((\,{\rm Im\,}c_1)\otimes_{\mathbb R}\mathbb C\right).$$ Then we get the desired inequality
\eqref{eqh11}.
We complete the proof of Theorem  \ref{theoremh11}.

\qquad
\section{Proof of Theorem \ref{theoremequality}}
Given a curve $B$ in a surface $X$ (non-zero effective divisor), we
denote by $B_{\rm red}$ the reduced part of  $B$.
Let $\Gamma_1,\cdots, \Gamma_\ell$ be all irreducible components of
$B$, and $\widetilde{\Gamma}_i\to \Gamma_i$ be the normalization. As
in the introduction, we define
   $$g(B)=\sum\limits_{i=1}^\ell g(\widetilde{\Gamma}_i),\qquad N_B=p_a(B)-p_a(B_{\rm red}).$$

   Let $q\in B$ be a singular point of $B_{\rm red}$,
we denote by $\mu_q(B)$ the Milnor's number of $B_{\rm red}$ at $q$,
and by $m_q(B)$ the multiplicity of $B_{\rm red}$ at $q$. Let
$\mu_B=\sum\limits_{q\in B}\mu_q(B)$, where
  $q$ runs over all singularities of $B_{\rm red}$.

Let $\sigma:\bar S\to S$ be the blowing-up at $q$,  $E$ the
exceptional curve, and $\bar B$ the strict transform of $B$ in $\bar
S$. Assume that $\bar{B}$ intersects $E$ at $r$ points
$q_1,\cdots,q_r$.

\begin{lemma}\label{lemmilnor}
 Assume that $B$ is a reduced curve with a singular point $q\in B$. Then
 \begin{enumerate}
\item[$(1)$] \  $\chi_{\rm top}(B)=2\chi(\mathcal{O}_{B})+\mu_B$.
\item[$(2)$] \   Let $m=m_q(B)$. Then
\begin{align*}
\mu_q(B)&=\sum\limits_{i=1}^{r}\mu_{q_i}(\bar{B})+m(m-1)-(r-1),\\
\mu_q(B)&=\sum\limits_{i=1}^{r}\mu_{q_i}(\bar{B}+E)+(m-1)(m-2)-1,
\end{align*}
\end{enumerate}
 \end{lemma}
(1) and (2) are proved in  \cite{Tan94} (Lemma~1.1 and  Lemma~1.3.)

\begin{definition}  A {\it partial
resolution} of the singularities of $B$ is a sequence of blowing-ups
$\sigma=\sigma_1\circ\sigma_2\circ \cdots \circ\sigma_r:$ $ \bar X
\to X$
$$
 (\bar
X,\sigma^* B)=(X_r,B_r)\overset{\sigma_r}\longrightarrow X_{r-1}
\overset{\sigma_{r-1}}\longrightarrow \cdots
\overset{\sigma_2}\longrightarrow (X_1,B_1)
\overset{\sigma_1}\longrightarrow (X_0,B_0)=(X,B),
$$
satisfying the following conditions:

(i) \ $B_{r,\textrm{red}}$ has at worst ordinary double points as
its singularities.

(ii) \,$B_i=\sigma_i^*B_{i-1}$ is the total transform of $B_{i-1}$.

\smallskip \noindent Furthermore,  $\sigma$ is called the {\it
minimal partial resolution} of the singularities of $B$ if
\smallskip

(iii) $\sigma_i$ is the blowing-up of $X_{i-1}$ at a singular point
$(B_{i-1,\textrm{red}},p_{i-1})$ which is not an ordinary double
point for any $i\leq r$. We denote by $m_{i+1}$ the multiplicity of $(B_{\rm red, i},p_i)$.
\end{definition}

In what follows, we always assume that the partial resolutions are
minimal, and we denote by $r=r(B)$ the number of blowing-ups in the
minimal resolution and by  $\overline{B}=B_r$.
Note that
$$\alpha(B)=\mu_{\overline{B}}-\ell(\overline{B})+1$$
 is the first Betti's number of
    the dual graph of $\overline{B}$ by Euler's formula. This number is determined uniquely by $B$.

From Lemma \ref{lemmilnor} and a straightforward computation, one
gets the following lemma.
 \begin{lemma}\label{lemNPamu} \begin{enumerate}
\item[$(1)$] $\mu_{\overline{B}}=\mu_B-\sum\limits_{i=1}^{r}(m_i-1)(m_i-2)+r$. \\
\item[$(2)$] $N_{\overline{B}}=N_B+\frac{1}{2}\sum\limits_{i=1}^{r}(m_i-1)(m_i-2)$.\\
\item[$(3)$]
$p_a(\overline{B}_{\rm red})=p_a(B_{\rm
red})-\frac{1}{2}\sum\limits_{i=1}^{r}(m_i-1)(m_i-2)$.
\end{enumerate}
 \end{lemma}

 For any fiber $F$,  we   define
$e_{F}= \chi_{\rm top}(F_{\rm red})-(2-2g).$
 \begin{corollary}\label{cortopology} \begin{enumerate}
 \item[$(1)$]  $e_F=2N_{F}+\mu_F=2N_{\overline{F}}+\mu_{\overline{F}}-r(F)$.
 \item[$(2)$]  $p_a(\overline
 F_{\rm red})\geq 0$ and $0\leq N_F\leq N_{\overline{F}}\leq g$.
\item[$(3)$]  $\alpha(F)=p_a(\overline{F}_{\rm red})-g(F)\geq 0$ and
   $g=g(F)+N_{\overline{F}}+\alpha(F)$. \end{enumerate}
  \end{corollary}
  \begin{proof}
It follows from Lemma  \ref{lemmilnor} (1) and Lemma \ref{lemNPamu}
immediately. Note that $\overline
 F_{\rm red}=\Gamma_1+\cdots+\Gamma_\ell$ is a connected nodal
 curve, $p_a(\Gamma_i)=g(\Gamma_i)+\alpha(\Gamma_i)$ and $$p_a(\overline
 F_{\rm
 red})=p_a(\Gamma_1)+\cdots+p_a(\Gamma_\ell)+\sum_{i<j}\Gamma_i\Gamma_j-\ell+1\geq
 0.$$
 $\alpha(F)=\alpha(\Gamma_1)+\cdots+\alpha(\Gamma_\ell)+\sum_{i<j}\Gamma_i\Gamma_j-\ell+1\geq
 0.$
  \end{proof}

 Let $\omega_{S/C}:=\omega_{S}\otimes (f_*\omega_C)^{\vee}$ be the relative canonical sheaf.
  The relative invariants of $f$ are defined as follows.
$$\chi_f:=\deg f_{*}\omega_{S/C},\qquad K_f^2:=\omega_{S/C}\cdot\omega_{S/C},\qquad  e_f:=\sum\limits_{F}e_{F}.$$
It is well known that
\begin{align*}\begin{cases}
K_f^2 =c_1^2(S)-8(g-1)(g(C)-1), &\\
\chi_f=\chi(\mathcal{O}_S)-(g-1)(g(C)-1),&\\
e_f= c_2(S)-4(g-1)(g(C)-1).&
\end{cases}
\end{align*}
By the definitions of $e_f, \chi_f$ and Hodge theory, one has
\begin{align*}
e_f&=2-4q+2p_g+h^{1,1}(S)-4(g-1)(b-1),\\
\chi_f &=1-q+p_g-(g-1)(b-1).
\end{align*}
Thus $ 2\chi_f-e_f=2q+2(g-1)(b-1)-h^{1,1}(S), $ i.e.,
\begin{align}\label{eq2chifef}
2\chi_f&=e_f+(g-q_f)(2b-2)+2q_fb+2-h^{1,1}(S).
\end{align}

On the other hand, from Corollary \ref{cortopology} and the fact
that $r(F_i)=\ell(\overline{F}_i)-\ell(F_i)$, we have
\begin{align}\label{eqef}
e_f&=\sum_{i=1}^se_{F_i} =\sum_{i=1}^s(2N_{\overline
F_i}+\mu_{\overline F_i}-r(F_i))=\sum_{i=1}^s(2N_{\overline
F_i}+\mu_{\overline F_i}-\ell(\overline F_i)+\ell(F_i)) \\ \notag &=
\sum_{i=1}^sN_{\overline F_i}+\sum_{i=1}^s(g-p_a(\overline F_{i,\rm
red})+\alpha(F_i) ) +\sum_{i=1}^s(\ell(F_i)-1) \\ \notag
&=\sum_{i=1}^sN_{\overline F_i}+ \sum_{i=1}^{s_1}(g-g(F_i))
+\sum_{i=1}^s(\ell(F_i)-1)         \\ \notag &=
\sum_{i=1}^{s}N_{\overline{F}_i}-\sum_{i=1}^{s_1}(g(F_i)-q_f)+\sum_{i=1}^{s}(\ell(F_i)-1)+(g-q_f)s_1.
\end{align}
Substitute    \eqref{eqef}  into   \eqref{eq2chifef}, we obtain the
equality in Theorem \ref{theoremequality}.

\qquad

\section{Applications}

The following lemma is due to Beauville (\cite{Bea81}, Lemma 1). The
original proof for the case when $C=\mathbb P^1$ works for the
general case. For the reader's convenience, we would like to recall
Beauville's proof.

\begin{lemma} \label{lemq>qf} {\rm (\cite{Bea81}) \ }
$g(F_i)\geq q_f$.
\end{lemma}
\begin{proof}
Let $\widetilde{F}$ be the normalization of $F$, and $\beta:
J(\widetilde{F})\to {\rm Alb}(S)$ be the natural map between the
jacobian $J(\widetilde{F})$ and the Albanese variety ${\rm Alb}(S)$.
Considering the abelian variety  $Q={\rm Alb}(S)/{\rm Im}\beta$, one
gets
 an  induced map $\bar{\alpha}:S\to Q$.   $\bar{\alpha}(F)$ is a point in $Q$ since $J(\widetilde{F})\to Q$ is zero.
 Therefore,  by the rigidity theorem, $\bar{\alpha}$ contracts all fibers of $f$.
So  $\bar{\alpha}$ has a  factorization through $f$.
\begin{equation*}
\raisebox{15mm}{\xymatrix{
S\ar[r]^{\bar{\alpha}}\ar[d]_{f} & Q   \\
C\ar[ur] & \\
}}
\end{equation*}
Since the image of $S$ in Alb$(S)$ generates Alb$(S)$, we see that
the image of $C$ in $Q$ generates $Q$. Thus we get a surjective map
$v:J(C)\to Q$.

We have the  following commutative diagram.
\begin{equation*}
\raisebox{15mm}{\xymatrix{
\widetilde{F}\ar[r]\ar[d]& J(\widetilde{F})\ar[d]_{\beta}\ar[dr] &\\
S\ar[r]^{\alpha}\ar[d]_{f}  & {\rm Alb}(S) \ar[r]\ar[d]^g & Q   \\
C\ar[r]_{j} &J(C)\ar[ur]_{v}&\\
}}
\end{equation*}
If $g(C)=0$, then  $Q$ is zero. Namely, $J(\widetilde{F})\to {\rm
Alb}(S)$ is surjective. Hence $g(F)=\dim J(\widetilde{F})\geq \dim
{\rm Alb}(S)=q$. If $g(C)>0$, then
 $v: J(C)\to  Q$ is surjective.
Thus
$$g(C)=\dim J(C)\geq \dim Q=q-\dim {\rm Im\,}\beta,$$
i.e., $\dim {\rm Im\,} \beta \geq q_f$. So
 $g(F)=\dim J(\widetilde{F})\geq \dim {\rm Im\,}\beta\geq q_f$.
\end{proof}


Similarly, from Corollary \ref{cortopology} and Lemma \ref{lemq>qf},  we get Corollary \ref{theoremAra2}.

\bigskip \noindent
 {\it Proof of Corollary \ref{corPgq1s14}:}
\bigskip

From \eqref{pg=} and the inequalities, we only need to prove that
$q\leq 1$.

Suppose that $q(S)\geq 2$. Because $p_g(S)=0$, $S$ is a ruled
surface. The Albanese map $\alpha: S\to \text{\rm Alb}(S)$ induces
the $\mathbb P^1$-fibration, and $B={\rm Im}\, \alpha$ is a curve of
genus $q$.

Because $g(F_i)=q\geq 2$, at least one irreducible component of
$\overline{F}_i$, say $\Gamma_1$,
 doesn't lie in the fibers of $\alpha$. So $g(\Gamma_1)\geq q=g(B)$ by Hurwitz formula. Thus
 $g(\Gamma_1)=q$ and $\Gamma_1$ is a section of $\alpha:S\to B$.  It implies $\Gamma_1$ is the unique horizonal
  irreducible component of $F_i$ since $g(F_i)=q$.  Since $F_i$ is semistable, $F_i=\Gamma_1+$ component contracted by $\alpha$.
  Let $F'$ (resp. $F$) be a general fiber of $\alpha$ (resp. $f$).
  One has $F'F_i=F'\Gamma_1=1$, and hence $F'F=1$. So $F\cong B$. Therefore $f$ is isotrivial. Since $f:S\to\mathbb P^1$
  is semistable, $f$ must be trivial, a contradiction. This
  completes the proof of Corollary \ref{corPgq1s14}.   \hfill\qed

\qquad

\section{Examples}

We will construct two semistable families $f:S\to\mathbb P^1$ with
$s_1=4$, $g=2$ and $q(S)=1$.

\begin{example}[\cite{Xia85}, Example~4.7]
Take six lines in $\BP^2$ as follows.
\begin{align*}
P_1:&\ X=0, & P_2: & \ X-Y=0,  & P_3: & \ Y=0,\\
Q_1: & \ 2X+Y-Z=0, & Q_2: &\  X+Y-Z=0,& Q_3:& \ X+2Y-Z=0.
\end{align*}
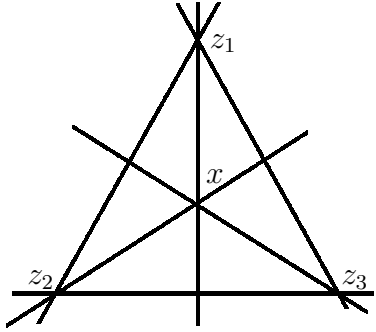
\unitlength 1mm 
\linethickness{1.2pt}
\ifx\plotpoint\undefined\newsavebox{\plotpoint}\fi 
\begin{picture}(191.75,58)(95,0)
\put(143.75,15.25){\line(1,0){48}}
\multiput(147.25,11.25)(.0337078652,.0600421348){712}{\line(0,1){.0600421348}}
\multiput(143,11)(.0523560209,.0337041885){764}{\line(1,0){.0523560209}}
\multiput(165.75,53.75)(.0337331334,-.0607196402){667}{\line(0,-1){.0607196402}}
\multiput(191,12.75)(-.0534741144,.033719346){734}{\line(-1,0){.0534741144}}
\put(168.5,54){\line(0,-1){43}}
\put(168.5,16.5){\line(0,-1){5.25}}
\put(170,48){$z_1$} \put(145.75,16.5){$z_2$} \put(187.5,16.5){$z_3$}
\put(169.5,30){$x$}
\end{picture}
This configration of 6 lines has 4 triple points $x$, $z_1$, $z_2$,
$z_3$, and 3 double points $y_1$, $y_2$ and $y_3$.  Their
coordinates are as follows.  $x= [0,0,1]$ and
\begin{align*}
y_1 =&[{1}/{2},0,1], & y_2 =&[{1}/{2},{1}/{2},1], &y_3=&[0, {1}/{2},1],\\
z_1 =&[1,0,1], & z_2=&[{1}/{3},{1}/{3},1], &z_3=&[0, 1,1]
\end{align*}

By Bertini's theorem, one can find an irreducible and reduced curve
$D$ of degree $4$ in $\BP^2$ satisfying the following conditions.
 \begin{enumerate}
\item[(1)] $D$ has ordinary  double points at $y_1$, $y_2$ and
$y_3$, and no other singular points.
\item[(2)] $D$ passes through $z_1$, $z_2$, $z_3$ and $x$.
\end{enumerate}

By blowing-up $\BP^2$ at $x$, we get a ruled surface $\varphi:P\to
\BP^1$. Thus we can construct a double cover $\pi:X\to P$ branched
along the curve $R=D+P_1+P_2+P_3+Q_1+Q_2+Q_3$. The double cover
gives us a semistable fibration $f:S\to \BP^1$ of genus $2$.

By a straightforward computation, we see that $f$ is a Lefschetz
pencil.  Furthermore, we have $K_S^2=-3$, $p_g(S)=0$, $q(S)=1$. Thus
$K_f^2=5$, $\chi_f=1$, $e_f=7$. So $f$ admits 7 singular fibers.
$K_f^2-2\chi_f=3$, this means that 3 singular fibers are not
irreducible, so $s_1=7-3=4$.
$$s_1=4, \hskip0.5cm s=7, \hskip0.5cm p_g=0, \hskip0.5cm q=1,  \hskip0.5cm h^{1,1}=5.$$
\end{example}

\begin{example}
Let ${\rm pr}_i:\BP^1\times \BP^1\to\mathbb P^1$ be the $i$-th
projection and $F_i^0$ be a general fiber of ${\rm pr}_i$  ($i=1,\, 2$). Let $B_0$ be a smooth irreducible curve of type $(2,1)$, i.e.,
$B_0\equiv 2F_1^0+F_2^0$.

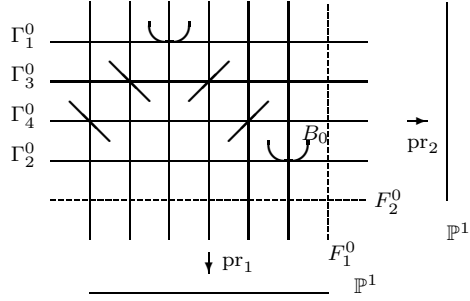
\begin{center}
\begin{picture}(50,125)(30, -45)
\multiput(0,-15)(15,0){6}{\line(0,1){90}}
\multiput(-15,60)(0,-15){4}{\line(1,0){120}}
\multiput(-15, 0)(3,0){40}{\line(1,0){2}}\multiput(90, -15)(0,3){30}{\line(0,1){2}}
\put(135, 0){\line(0,1){75}}\put(0,-35){\line(1,0){90}}
\put(45,-20){\vector(0,-1){7.5}}\put(120,30){\vector(1,0){7.5}}

\put(-30, 15){{\tiny $\Gamma_2^{0}$}}\put(-30,30){{\tiny $\Gamma_4^{0}$}}\put(-30, 45){{\tiny $\Gamma_3^{0}$}}
\put(-30, 60){{\tiny $\Gamma_1^{0}$}}\put(100, -35){{\tiny $\mathbb{P}^1$}}
\put(135, -15){{\tiny $\mathbb{P}^1$}}\put(50, -25){{\tiny ${\rm pr}_1$}}
\put(120, 20){{\tiny ${\rm pr}_2$}}\put(90, -22.5){{\tiny $F_1^0$}}\put(107.5,  -2.5){{\tiny $F_2^0$}}
\put(80, 23){{\tiny $B_0$}}
\thicklines
\put(30,67.5){\oval(15,15)[b]}\put(75,22.5){\oval(15, 15)[b]}
\multiput(7.5,22.5)(15,15){2}{\line(-1,1){15}}\multiput(52.5,22.5)(-15,15){2}{\line(1, 1){15}}
 \end{picture}
\end{center}

There are two fibers $\Gamma_1^0$ and
$\Gamma_2^0$ of ${\rm pr}_2$ tangent to $B_0$.
 By choosing two general fibers $\Gamma_3^0$ and  $\Gamma_4^0$ of ${\rm
 pr}_2$,
 one can construct a double cover  $\pi_0:E\times\BP^1\to
\BP^1\times \BP^1$ ramified over
$B_{\pi_0}=\sum\limits_{i=1}^4\Gamma_i^{0}$, where $E$ is
 the pullback elliptic curve of a general fiber of ${\rm pr}_1$.

Let
$$B_{\pi}:=\pi_0^*B_0 \equiv 2F_1+2F_2,\qquad \pi_0^*\Gamma_i^0=2\Gamma_i,\qquad  L=F_1+\Gamma_3, $$
where $F_i$ are   general fibers of
the $i$-th projection of $E\times \BP^1$. Thus $B\equiv 2L$. We can construct a double cover $\pi:S_0\to E\times
 \BP^1$ branched along $B_\pi$. Thus we get a fibration $f:S\to \BP^1$ of genus $2$.

  Now we claim that $f$ has six singular fibers, and four of them have non-compact Jacobians.
Let $E_1$ and $E_2$ be the elliptic fibers of  ${\rm pr}_1:
E\times\BP^1\to\BP^1$ such that the image fiber $\pi_0(E_i)$ passes
through the tangent point $B_0\cap \Gamma_i^0$ ($i=1,2$). Let
$E_3,E_4$ (resp.  $E_5,E_6$) are other elliptic fibers whose  image
fibers  pass through the intersection points  of $B_{0}\cap
\Gamma_3^0$ (resp. $B_{0}\cap \Gamma_4^0$). Take $p_i=E_i\cap
\Gamma_i$, $p_{i+2}=E_{i+2}\cap \Gamma_3$, $p_{i+4}=E_{i+4}\cap
\Gamma_4$ ($i=1,2$) and $q_j=E_j\cap \Gamma_3$ ($j=1,\cdots,6$).

 
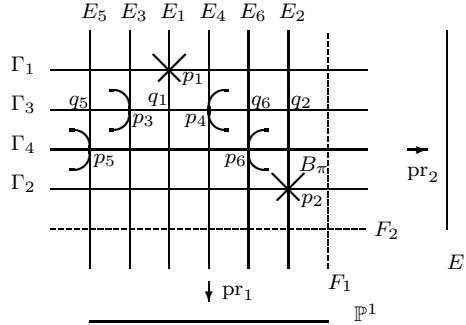
\begin{center}
\begin{picture}(50,135)(30, -45)
\multiput(0,-15)(15,0){6}{\line(0,1){90}}
\multiput(-15,60)(0,-15){4}{\line(1,0){120}}
\multiput(-15, 0)(3,0){40}{\line(1,0){2}}\multiput(90, -15)(0,3){30}{\line(0,1){2}}
\put(135, 0){\line(0,1){75}}\put(0,-35){\line(1,0){90}}
\put(45,-20){\vector(0,-1){7.5}}\put(120,30){\vector(1,0){7.5}}

\put(-30, 15){\tiny $\Gamma_2$}\put(-30,30){\tiny $\Gamma_4$}\put(-30, 45){\tiny $\Gamma_3$}
\put(-30, 60){\tiny $\Gamma_1$}\put(100, -35){\tiny $\mathbb{P}^1$}
\put(135, -15){\tiny $E$}\put(50, -25){\tiny ${\rm pr}_1$}
\put(120, 20){\tiny ${\rm pr}_2$}\put(90, -22.5){\tiny $F_1$}\put(107.5,  -2.5){\tiny $F_2$}
\put(79, 22){\tiny $B_{\pi}$}
\put(-3.5, 80){\tiny $E_5$}\put(11.5, 80){\tiny $E_3$}\put(26.5, 80){\tiny $E_1$}
\put(41.5, 80){\tiny $E_4$}\put(56.5, 80){\tiny $E_6$}\put(71.5, 80){\tiny $E_2$}
\put(1.5,25){\tiny $p_5$}\put(-8,47){\tiny $q_5$}\put(16,40){\tiny $p_3$}\put(22,48){\tiny $q_1$}
\put(36,40){\tiny $p_4$}\put(61,47){\tiny $q_6$}\put(76,47){\tiny $q_2$}\put(51,25){\tiny $p_6$}
\put(35,56){\tiny $p_1$}\put(80,10){\tiny $p_2$}
\thicklines
\multiput(34,56)(45,-45){2}{\line(-1,1){10}}\multiput(26,56)(45,-45){2}{\line(1,1){10}}
\multiput(-7.5,30)(15,15){2}{\oval(15,15)[r]}\multiput(67.5,30)(-15,15){2}{\oval(15,15)[l]}
 \end{picture}
\end{center}

  Since $p_i\ne q_i$  and
  $$2p_i=B_{\pi}\mid_{E_i}\equiv 2L|_{E_i}=2q_i,\qquad i=1,2,$$
  $\pi^{-1}(E_i)$ is irreducible ($i=1,2$). So the fiber $\widetilde{F}_i$ of $f$ corresponding to $\pi^{-1}(E_i)$
 ($i=1,2$)  can be written as $\widetilde{F}_i=C_1+C_2$ with $C_1C_2=2$, where $C_1$ is a smooth elliptic curve and
    $C_2$ is a $(-2)$-curve.

  $p_3=q_3$ (resp. $p_4=q_4$) implies that $\pi^{-1}E_3$ (resp. $\pi^{-1}E_4$)
  is  reducible. So  the corresponding fiber $\widetilde{F}_3$ (resp. $\widetilde{F}_4$) of $f$ is a nodal curve
  $C_1+C_2$ with $C_1C_2=1$ where $C_1, C_2$ are smooth  elliptic curves.

Similarly, one can check that $\pi^{-1}E_5$ and $\pi^{-1}E_6$ are
irreducible. Thus the corresponding fiber $\widetilde{F}_5$
 and $\widetilde{F}_6$ are singular elliptic curves with only one node.

By a straightforward computation, one has
$$s_1=4, \hskip0.5cm s=6, \hskip0.5cm p_g=0, \hskip0.5cm q=1,  \hskip0.5cm h^{1,1}=6.$$
\end{example}
\begin{remark}
The third author proves in \cite{Yu00} that if $s_1=4$ and $s=5$, then
$g=2$.
\end{remark}

\bigskip\noindent
\textbf{Acknowledgements:} The authors would like to thank Dr. Xin L\"u for
pointing out an error in the original proof of Lemma \ref{Claim 4.}.

\qquad

\end{document}